\documentclass{article}%
\usepackage{amsmath,array}
\usepackage{amsfonts}
\usepackage{amssymb}
\usepackage{blindtext}
\usepackage{enumitem}
\usepackage{fullpage}
\usepackage{graphicx}
\usepackage{pstricks}
\usepackage{pstricks-add}
\usepackage{pst-node}
\usepackage{pst-plot}
\usepackage{tikz,verbatim,xcolor,caption,float}%
\usepackage{amsmath}%
\providecommand{\U}[1]{\protect\rule{.1in}{.1in}}
\newtheorem{theorem}{Theorem}

\newtheorem{claim}[theorem]{Claim}

\newtheorem{corollary}[theorem]{Corollary}

\newtheorem{example}[theorem]{Example}

\newtheorem{lemma}[theorem]{Lemma}

\newtheorem{remark}[theorem]{Remark}

\newenvironment{proof}{\noindent{\em Proof:}}{$\Box$~\\}
\allowdisplaybreaks
\begin{document}

\title{A New Proof of Meissner's Optimal Bound on the Degree of a Poincar\'e Multiplier and an Improved Optimal Degree Multiplier}
\author{Brittany Riggs}
\maketitle

\begin{abstract}
Let $f$ be a monic univariate polynomial. We say
that $f$ is \emph{positive\/} if~$f(x)$ is positive over all $x > 0$. If all
the coefficients of $f$ are non-negative, then $f$ is trivially positive. In~
1883, Poincar\'e proved that$f$ is positive if and only if there exists a monic
polynomial $g$ such that all the coefficients of $gf$ are non-negative. Such
polynomial $g$ is called a \emph{Poincar\'e multiplier\/} for the positive
polynomial $f$. Of course one hopes to find a multiplier with smallest degree.
In 1911, Meissner provided such a bound for quadratic polynomials. In this paper, we provide a linear algebra proof of Meissner's optimal bound and compare an improved optimal degree Poincar\'e multiplier to one provided by Meissner.
\end{abstract}

\section{Introduction}

Let $f$ be a monic univariate polynomial. We say that $f$ is \emph{positive\/}
if~$f(x)$ is positive over all $x>0$.

If all the coefficients of $f$ are non-negative, then obviously $f$ is
positive, but the converse is not true. Consider the following small (toy)
examples:
\begin{align*}
f_{1}  &  =x^{4}+x^{3}+10x^{2}+2x+10\\
f_{2}  &  =x^{4}-x^{3}-10x^{2}-2x+10\\
f_{3}  &  =x^{4}-x^{3}+10x^{2}-2x+10
\end{align*}
\noindent Note that all the$\ $coefficients of $f_{1}$ are non-negative. Thus
it is trivial to see that $f_{1}$ is positive. However $f_{2}$ and $f_{3}$
have negative coefficients, and thus it is not obvious whether they are
positive or not. It turns out that $f_{2}$ is not positive and $f_{3}$ is positive.

In 1883, Poincar\'e \cite{P83} proposed and proved a modification of the
converse: if $f$ is positive, then there exists a monic polynomial $g$ such
that all the coefficients of $gf$ are non-negative. Such polynomial $g$ is
called a \emph{Poincar\'{e} multiplier} for the positive polynomial $f$. For
the above examples, we have:

\begin{itemize}
\item Since $f_{2}$ is not positive, there is no Poincar\'{e} multiplier for
$f_{2}$.

\item Since $f_{3}$ is positive, there is a Poincar\'{e} multiplier for
$f_{3}$. For instance, let $g=x+1$. Then
\[
gf_{3}=\left(  x+1\right)  \left(  x^{4}-x^{3}+10x^{2}-2x+10\right)
=x^{5}+9x^{3}+8x^{2}+8x+10
\]
Note that all the coefficients of $gf_{3}$ are non-negative.
\end{itemize}

In 1911, Meissner \cite{M11} (an English translation is provided in \cite{W23}
and an alternate explanation in \cite{C18}) provided a shape for this
multiplier for an arbitrary positive quadratic polynomial $f=x^{2}%
-2r\cos(\theta)x+r^{2}$ with non-real roots $re^{\pm i\theta}$ and gave an optimal degree for this multiplier. The following $g$ is a Poincar\'{e}
multiplier for $f$:%

\begin{equation}
\label{eq:Meissner}g_{M}=\sum_{i=0}^{s} \dfrac{r^{2}\sin(i+1)\theta}%
{\sin(\theta)}\left(  \dfrac{x}{r}\right)  ^{i}%
\end{equation}
where
\begin{align*}
s  &  =\left\lceil \dfrac{\pi}{\theta} \right\rceil -2.
\end{align*}

This bound on the degree of a Poincar\'{e} multiplier for a quadratic
polynomial can easily be extended to higher degrees. Naturally, we wonder if
the Meissner bound on the degree of the Poincar\'{e} multiplier is
optimal in any way for higher degree polynomials. The answer to this question
is left to a different paper.

In working to prove the optimality of Meissner's bound for arbitrary
polynomials, we developed a linear algebra approach to this subject that led to a new
proof of Meissner's optimal bound for quadratic polynomials and a new optimal degree Poincar\'{e}
multiplier.  While there is no unique multiplier of a given degree, we will demonstrate that the new multiplier is less than or equal to Meissner's using two different partial orderings.

In Section \ref{sec:LA} of this paper, we will develop the linear algebra
framework. We use it to prove Meissner's bound (Section \ref{sec:bound})
and a new multiplier (Section \ref{sec:new}). Finally, in Section
\ref{sec:compare}, we study Meissner's multiplier through the new linear
algebra lens and compare it to our Poincar\'{e} multiplier. This section
highlights the connections between this work and Meissner's as well as
demonstrates how we have improved upon Meissner's original multiplier.

\section{Main Results}

\begin{theorem}
[Meissner's Optimal Bound for Degree 2 \cite{M11}]\label{thm:optimal2} Let
$f=x^{2}+a_{1}x+a_{0}\in\mathbb{R}\left[  x\right]  $ be without real roots.
Then the smallest degree of non-zero $g\in\mathbb{R}\left[  x\right]  $ such
that $\operatorname*{coeffs}\left(  gf\right)  \geq0$ is
\[
s=\left\lceil \frac{\pi}{\theta}\right\rceil -2
\]
where $\alpha=re^{i\theta}$ is the non-real root of $f$ with positive
imaginary part.
\end{theorem}

\begin{theorem}
[Riggs Multiplier for Degree 2]\label{thm:certificate2}Let $f=x^{2}%
+a_{1}x+a_{0}\in\mathbb{R}\left[  x\right]  $ be without real roots. Then a
witness for $\operatorname*{coeffs}\left(  gf\right)  \geq0$ is given by%
\[
g_{R}=\left\vert
\begin{array}
[c]{ccccc}%
a_{2} &  &  &  & x^{0}\\
a_{1} & \ddots &  &  & \vdots\\
a_{0} &  & \ddots &  & \vdots\\
& \ddots &  & \ddots & \vdots\\
&  & a_{0} & a_{1} & x^{s}%
\end{array}
\right\vert
\]
where
\[
s=\left\lceil \frac{\pi}{\theta}\right\rceil -2
\]
and $\alpha=re^{i\theta}$ is the non-real root of $f$ with positive
imaginary part.
\end{theorem}

\begin{remark}
The linear algebra approach naturally leads to the certificate provided in
Theorem \ref{thm:certificate2}. This certificate for the multiplier is unique
from the one provided originally by Meissner in Equation \ref{eq:Meissner}.
\end{remark}

\begin{remark}
The coefficients of Meissner's multiplier are given in terms of the radius and
angle of the roots of the quadratic. We are naturally interested in
finding similar expressions for the coefficients of $g_{R}$.
\end{remark}

\begin{theorem}
[Coefficients of Riggs Multiplier in terms of the Root]\label{thm:coeffsHRM2}
The coefficients of $g_{R}$ are given by
\begin{align*}
b_{s-j}  &  =r^{j} \cdot\dfrac{\sin(j+1)\theta}{\sin(\theta)}\ \ \ \ \text{for }j=0,\hdots,s
\end{align*}
where $re^{\pm i\theta}$ are the non-real roots of $f$.
\end{theorem}

\begin{remark}
\label{rem:coeff} The coefficients of $g_{R}$ can also be expressed as
follows:
\begin{align*}
b_{i}  &  =r^{s-i} \cdot\dfrac{\sin(s-i+1)\theta}{\sin(\theta)}\ \ \ \ \text{for }i=0,\hdots,s
\end{align*}
We have chosen the form in Theorem \ref{thm:coeffsHRM2} as the proof flows
more naturally.
\end{remark}

\begin{example}
Let $f=x^{2}-2x+2$, which has non-real roots $\sqrt{2}e^{\pm\frac{\pi}{4}i}$.
Note
\[
s=\left\lceil \frac{\pi}{\frac{\pi}{4}}\right\rceil -2=2.
\]
The Meissner multiplier, $g_{M}$, is given by
\begin{align*}
g_{M}  &  =\sum_{i=0}^{2} \dfrac{\sqrt{2}^{2} \sin\left( (i+1)\dfrac{\pi}%
{4}\right) }{\sin\left( \dfrac{\pi}{4}\right) }\left( \dfrac{x}{\sqrt{2}%
}\right) ^{i}\\
&  = \dfrac{\sin\left( \dfrac{3\pi}{4}\right) }{\sin\left( \dfrac{\pi}%
{4}\right) }x^{2}+\dfrac{\sqrt{2}\sin\left( \dfrac{2\pi}{4}\right) }%
{\sin\left( \dfrac{\pi}{4}\right) }x+\dfrac{\sqrt{2}^{2}\sin\left( \dfrac{\pi
}{4}\right) }{\sin\left( \dfrac{\pi}{4}\right) }\\
&  =x^{2}+2x+2
\end{align*}
Note that
\[
g_{M}\, f=\left(  x^{2}+2x+2\right)  \left(  x^{2}-2x+2\right)  =x^{4}+4.
\]
The new multiplier, $g_{R}$, is given by
\begin{align*}
g_{R}  &  =\left|
\begin{array}
[c]{rrr}%
1 & 0 & x^{0}\\
-2 & 1 & x^{1}\\
2 & -2 & x^{2}%
\end{array}
\right| \\
&  =x^{2}+2x+2
\end{align*}
Note, in this case, that $g_{M}=g_{R}$.
\end{example}

\begin{example}
\label{ex:notequal} Let $f=x^{2}-2\cos\left( \dfrac{2\pi}{7}\right) x+1$,
which has non-real roots $1e^{\pm\frac{2\pi}{7}i}$. Note
\[
s=\left\lceil \frac{\pi}{\frac{2\pi}{7}}\right\rceil -2=2.
\]
The Meissner multiplier, $g_{M}$, is given by
\begin{align*}
g_{M}  &  =\sum_{i=0}^{2} \dfrac{1^{2} \sin\left( (i+1)\dfrac{2\pi}{7}\right)
}{\sin\left( \dfrac{2\pi}{7}\right) }\left( \dfrac{x}{1}\right) ^{i}\\
&  = \dfrac{\sin\left( \dfrac{6\pi}{7}\right) }{\sin\left( \dfrac{2\pi}%
{7}\right) }x^{2}+\dfrac{\sin\left( \dfrac{4\pi}{7}\right) }{\sin\left(
\dfrac{2\pi}{7}\right) }x+\dfrac{\sin\left( \dfrac{2\pi}{7}\right) }%
{\sin\left( \dfrac{2\pi}{7}\right) }\\
&  \approx0.555x^{2}+1.247x+1
\end{align*}
Note that
\begin{align*}
g_{M}\, f  &  =\left( 0.555x^{2}+1.247x+1\right) \left( x^{2}-2\cos\left(
\dfrac{2\pi}{7}\right) x+1\right) \\
&  \approx0.555x^{4}+0.555x^{3}+1.
\end{align*}
The new multiplier, $g_{R}$, is given by
\begin{align*}
g_{R}  &  =\left|
\begin{array}
[c]{ccc}%
1 & 0 & x^{0}\\
-2\cos\left( \dfrac{2\pi}{7}\right)  & 1 & x^{1}\\
1 & -2\cos\left( \dfrac{2\pi}{7}\right)  & x^{2}%
\end{array}
\right| \\
&  =x^{2}+2\cos\left( \dfrac{2\pi}{7}\right) x+\left( 4\cos^{2}\left(
\dfrac{2\pi}{7}\right) -1\right) \\
&  \approx x^{2}+1.247x+0.555
\end{align*}
Note that
\begin{align*}
g_{HR}\, f  &  =\left(  x^{2}+1.247x+0.555\right)  \left( x^{2}-2\cos\left(
\dfrac{2\pi}{7}\right) x+1\right) \\
&  \approx x^{4}+0.555x+0.555.
\end{align*}
Note, in this case, that $g_{M} \neq g_{R}$.
\end{example}

\section{Reframing with Linear Algebra}

\label{sec:LA}

Meissner's proof of the multiplier in Equation \ref{eq:Meissner} and the bound in Theorem \ref{thm:optimal2} relied on a geometric interpretation of the problem that is not fleshed out in all its detail.  Here, we use linear algebra to provide a fully explicit proof. 

Let $f\in\mathbb{R}\left[  x\right]  $ be monic without losing generality.
Assume that $f$ does not have positive real root. The problem is to find
non-zero $g\in\mathbb{R}\left[  x\right]  $ such that $\operatorname*{coeffs}%
\left(  gf\right)  \geq0$. Let%
\begin{align*}
f &  =a_2x^{2}+a_{1}x^{1}+a_{0}x^{0}\\
g &  =b_{s}x^{s}+\cdots+b_{0}x^{0}%
\end{align*}
where $a_{2}=1$ and $b_{s}=1$ (we can also assume $g$ is monic without losing generality). In the steps below, we will study the coefficients of $gf$ for arbitrary $s$, but demonstrate the matrices using $s=3$ to provide clarity.  These steps extend directly to higher $s$-values.

Note
\begin{align*}
\operatorname*{coeffs}\left(  gf\right)    & =\underset{b\in\mathbb{R}%
^{1\times\left(  s+1\right)  }}{\underbrace{\left[
\begin{array}
[c]{cccc}%
b_{0} & b_{1} & b_{2} & 1
\end{array}
\right]  }}\underset{A_{s}\in\mathbb{R}^{\left(  s+1\right)  \times\left(
s+3\right)  }}{\underbrace{\left[
\begin{array}
[c]{cccccc}%
a_{0} & a_{1} & 1 &  &  & \\
& a_{0} & a_{1} & 1 &  & \\
&  & a_{0} & a_{1} & 1 & \\
&  &  & a_{0} & a_{1} & 1
\end{array}
\right]  }}\end{align*}
This is easy to verify, e.g. the coefficient in $gf$ of $x^0$ is $b_0a_0$, the coefficient of $x^1$ is $b_0a_1+b_1a_0$, etc.

It proves useful to consider the first two columns of $A_s$ independently from the remaining $(s+1)$ columns.  Below, we adjoin the two sub-matrices, give a new label to each, then manipulate the expression in an effort to reduce our search space to fewer columns.

\begin{align*} \operatorname*{coeffs}\left(  gf\right) & =b\left[  \underset{L_{s}\in\mathbb{R}^{\left(  s+1\right)  \times
2}}{\underbrace{\left[
\begin{array}
[c]{cc}%
a_{0} & a_{1}\\
& a_{0}\\
& \\
&
\end{array}
\right]  }}\Bigg|\underset{R_{s}\in\mathbb{R}^{\left(  s+1\right)  \times\left(
s+1\right)  }}{\underbrace{\left[
\begin{array}
[c]{cccc}%
1 &  &  & \\
a_{1} & 1 &  & \\
a_{0} & a_{1} & 1 & \\
& a_{0} & a_{1} & 1
\end{array}
\right]  }}\right]  \\
& =bR_{s}R_{s}^{-1}\left[  L_{s}|R_{s}\right] \ \ \ \ \text{(we multiply by } I=R_sR_s^{-1}\text{)} \\
& =bR_{s}\left[  R_{s}^{-1}L_{s}|R_{s}^{-1}R_{s}\right] \ \ \ \ \text{(we distribute } R_s^{-1} \text{ into the adjoined matrix)}  \\
& =\underset{c\in\mathbb{R}^{1\times\left(  s+1\right)  }}{\underbrace{bR_{s}%
}}\left[  \underset{T_{s}\in\mathbb{R}^{\left(  s+1\right)  \times
2}}{\underbrace{R_{s}^{-1}L_{s}}}|I_{s+1}\right]  \\
& =\left[  cT_{s}\ |\ c\right]
\end{align*}
Thus%
\begin{align*}
& \underset{g\neq0,\ \deg\left(  g\right)  \leq s}{\exists}%
\ \operatorname*{coeffs}\left(  gf\right)  \geq0\\
\Longleftrightarrow \ & \underset{b\neq0}{\exists}\ \left[  cT_{s}\ |\ c\right]
\geq0\\
\Longleftrightarrow\ & \underset{c\neq0}{\exists}\ \left[  cT_{s}\ |\ c\right]
\geq0\\
\Longleftrightarrow \ & \underset{c\neq0}{\exists}\ cT_{s}\geq0\text{ and }c\geq0
\end{align*}

Where before we were studying a $1 \times (s+1)$ vector and an $(s+1) \times (s+3)$ matrix, now we are interested in a non-negative $1 \times (s+1)$ vector and an $(s+1) \times 2$ matrix.  We will study $T_s$.

\[\begin{array}{crl}
 & T_{s}  & =R_s^{-1}L_s \\
\iff & R_sT_s & =L_s \\
\iff &
\left[
\begin{array}
[c]{cccc}%
1 &  &  & \\
a_{1} & 1 &  & \\
a_{0} & a_{1} & 1 & \\
& a_{0} & a_{1} & 1
\end{array}
\right]\left[
\begin{array}
[c]{cc}%
T_{00} & T_{01}\\
T_{10} & T_{11}\\
T_{20} & T_{21} \\
T_{30} & T_{31}
\end{array}
\right] & = \left[
\begin{array}
[c]{cc}%
a_{0} & a_{1}\\
& a_{0}\\
& \\
&
\end{array}
\right]
\end{array}\]
Let's examine the entries $T_{k0}$.  Note
\[\begin{array}{crl} & T_{00} & =a_0=r^2 \\
& a_1T_{00}+T_{10} & =0 \\
\iff & T_{10} & =-a_1T_{00}=2r^3\cos(\theta) \\
& a_0T_{00}+a_1T_{10}+T_{20} & =0 \\
\iff & T_{20} & =-a_0T_{00}-a_1T_{10} \\
& & =r^4\left(-1+4\cos^2(\theta)\right) \\
& & =r^4\left(-1+4(1-\sin^2(\theta)\right) \\
& & =r^2\left(3-4\sin^2(\theta)\right) \end{array}\]
Note that we can rewrite these entries as follows, using multiple angle identities (see Section 6.5.10 of \cite{Z18}):
\begin{align*} T_{00} & =r^2 \cdot \dfrac{\sin(\theta)}{\sin(\theta)} \\
T_{10} & =r^3 \cdot \dfrac{\sin(2\theta)}{\sin(\theta)} \\
T_{20} & =r^4 \cdot \dfrac{\sin(3\theta)}{\sin(\theta)} \end{align*}

\begin{claim}\label{cl:k0} We have
$$T_{k0}=r^{k+2} \cdot  \dfrac{\sin(k+1)\theta}{\sin(\theta)}.$$
\end{claim}

Now we'll examine the entries $T_{k1}$, using the same identities as above. 
\[\begin{array}{crl} & T_{01} & =a_1=-2r\cos(\theta)=-r \cdot \dfrac{\sin(2\theta)}{\sin(\theta)} \\
& a_1T_{01}+T_{11} & =a_0 \\
\iff & T_{11} & =a_0-a_1T_{01}=r^2-4r^2\cos^2(\theta)=-r^2 \cdot \dfrac{\sin(3\theta)}{\sin(\theta)} \end{array}\]

\begin{claim}\label{cl:k1} We have
$$T_{k1}=-r^{k+1} \cdot \dfrac{\sin(k+2)\theta}{\sin(\theta)}.$$
\end{claim}

\section{Proof of Meissner's Optimal Bound for Degree 2 (Theorem
\ref{thm:optimal2})}\label{sec:bound}

First, we will combine Claims \ref{cl:k0} and \ref{cl:k1} into Lemma \ref{lem:Ast2} and prove it.  Then we will prove there exists a multiplier with degree equal to the bound given in Theorem \ref{thm:optimal2}.  Finally, we will prove Theorem \ref{thm:optimal2}, that this bound is indeed optimal.

\begin{lemma}
\label{lem:Ast2}Let $f\in\mathbb{R}[x]$ be such that $\deg(f)=2$ without real
roots. Let the roots be $re^{\pm i\theta}$. Then we have%
$$T_{k0}=+r^{k+2}\dfrac{\sin\left(  k+1\right)  \theta}{\sin\theta} \ \ \ \ \text{and}\ \ \ \ T_{k1}=-r^{k+1}\dfrac{\sin(k+2)\theta}{\sin\theta}.$$

\end{lemma}

\begin{proof}
We will prove by induction on $k$, much like in the exposition above.  Note that

\[\begin{array}{crcl} & T_{s}  & = & R_s^{-1}L_s \\
\iff & \left[
\begin{array}
[c]{ccccc}%
1 &  &  & \\
a_{1} & 1 &  & \\
a_{0} & a_{1} & 1 & \\
& \ddots & & \ddots & \\
&& a_{0} & a_{1} & 1
\end{array}
\right]\left[
\begin{array}
[c]{cc}%
T_{00} & T_{01}\\
T_{10} & T_{11}\\
T_{20} & T_{21} \\
\vdots & \vdots \\
T_{s0} & T_{s1} 
\end{array}
\right] & = & \left[
\begin{array}
[c]{cc}%
a_{0} & a_{1}\\
& a_{0}\\
& \\
& \\
& 
\end{array}
\right] \end{array}\]

\begin{description}
\item[\textsf{Base Cases:}] $k=0$ and $k=1$ \\
For $T_{k0}$ we have:
\begin{align*}
T_{00} & =a_0=r^2=r^2 \cdot \dfrac{\sin(\theta)}{\sin(\theta)} \\
T_{10} & =-a_1T_{00}=2r^3\cos(\theta)=r^3 \cdot \dfrac{\sin(2\theta)}{\sin(\theta)}  \end{align*}
For $T_{k1}$ we have:
\begin{align*} T_{01} & =a_1=-2r\cos(\theta)=-r \cdot \dfrac{\sin(2\theta)}{\sin(\theta)} \\
T_{11} & =a_0-a_1T_{01}=r^2-4r^2\cos^2(\theta)=-r^2 \cdot \dfrac{\sin(3\theta)}{\sin(\theta)}
\end{align*}
The base cases are proved.

\item[\textsf{Inductive Hypothesis:}] Assume for $k \geq 0$ we have
$$T_{k0}=+r^{k+2}\dfrac{\sin\left(  k+1\right)  \theta}{\sin\theta} \ \ \ \ \text{and}\ \ \ \ T_{k1}=-r^{k+1}\dfrac{\sin(k+2)\theta}{\sin\theta}.$$

\item[\textsf{Induction Step:}] Show that 
\[%
\begin{array}
[c]{rcl}%
T_{k+1,0} & = & +r^{(k+1)+2}\dfrac{\sin\left(  (k+1)+1\right)  \theta}{\sin\theta} =+r^{k+3}\dfrac{\sin\left(  k+2\right)  \theta}{\sin\theta} \vspace{0.2cm}\\
T_{k+1,1} & = & -r^{(k+1)+1}\dfrac{\sin((k+1)+2)\theta}{\sin\theta}=-r^{k+2}\dfrac{\sin(k+3)\theta}{\sin\theta}
\end{array}
\]
In general, we have
$$T_{k+1,0}=-a_0T_{k-1,0}-a_1T_{k0}\ \ \ \ \text{and}\ \ \ \ T_{k+1,1}=-a_0T_{k-1,1}-a_1T_{k1}$$
Then
\begin{align*} T_{k+1,0} & =-r^2 \cdot r^{k+1} \cdot \dfrac{\sin\left(  k \theta\right) }{\sin\theta}+2r\cos(\theta) \cdot r^{k+2} \cdot \dfrac{\sin\left(  k+1\right)  \theta}{\sin\theta}  \\
 & =r^{k+3} \left(\dfrac{2\cos(\theta)\sin(k+1)\theta-\sin(k\theta)}{\sin(\theta)}\right) \\
& =r^{k+3} \cdot \dfrac{\sin(k+2)\theta}{\sin(\theta)}\ \ \ \text{ by the multiple angle formula in Section 6.5.10 of \cite{Z18}} \\
T_{k+1,1} & =-r^2 \cdot -r^{k} \cdot \dfrac{\sin(k+1)\theta}{\sin\theta}+2r\cos(\theta) \cdot -r^{k+1} \cdot \dfrac{\sin(k+2)\theta}{\sin\theta} \\
& =-r^{k+2}\left(\dfrac{2\cos(\theta)\sin(k+2)\theta-\sin(k+1)\theta}{\sin(\theta)}\right) \\
& =-r^{k+2} \cdot \dfrac{\sin(k+3)\theta}{\sin(\theta)}\ \ \ \text{ by the multiple angle formula in Section 6.5.10 of \cite{Z18}} \end{align*}
Hence, the induction step is proved.
\end{description}
\end{proof}

In the lemma below, we prove the existence of a multiplier with degree equal to the bound in \\ Theorem \ref{thm:optimal2}. \bigskip

\begin{lemma}
\label{lem:d2} Let $f\in\mathbb{R}[x]$ be such that $\deg(f)=2$ without real
roots. Let $re^{\pm i\theta}$ be the
roots of $f$ and
\[
s=\left\lceil \frac{\pi}{\theta}\right\rceil -2.
\]
Then
\[
\underset{g\neq0,\ \deg\left(  g\right)  \leq s}{\exists}%
\ \operatorname*{coeffs}\left(  gf\right)  \geq0.
\]

\end{lemma}

\begin{proof}
Recall from Section 3 we have 
$$\underset{g\neq0,\ \deg\left(  g\right)  \leq s}{\exists
}\ \operatorname*{coeffs}\left(  gf\right)  \geq0 \ \ \iff \ \ \underset{c\neq0}{\exists}\ cT_{s}\geq0\text{ and }c\geq0$$
A simple non-zero and non-negative $c$-vector is $\left[\begin{array}{cccc} 0 & \cdots & 0 & 1 \end{array}\right]$.  We will verify that $cT_{s}\geq0$ holds for this $c$-vector and $s$-value.

\begin{align*}
& \ \ \ cT_{s}\geq0 \ \ \wedge \ \ c=\left[\begin{array}{cccc} 0 & \cdots & 0 & 1 \end{array}\right] \\
\iff  &  \ \ \ T_{s0},T_{s1}\geq0\\
\Longleftarrow\, \ &  \ \ \ \sin\left(  s+1\right)  \theta>0\;\wedge
\;\sin\left(  s+2\right)  \theta\leq0\ \text{(from Lemma \ref{lem:Ast2})}\\
\Longleftarrow\, \  &  \ \ \ 0<(s+1)\theta<\pi\;\wedge\;\pi\leq(s+2)\theta
<2\pi\\
\iff &  \ \ \ s<\dfrac{\pi}{\theta}-1\;\wedge\;s\geq\dfrac{\pi
}{\theta}-2\\
\iff &  \ \ \ \dfrac{\pi}{\theta}-2\leq s <\dfrac{\pi}{\theta
}-1\\
\Longleftarrow\, \  &  \ \ \ s=\left\lceil \dfrac{\pi}{\theta}\right\rceil -2
\end{align*}

\end{proof}

Finally, we prove that the bound on the degree of the multiplier given in Theorem \ref{thm:optimal2} is optimal. \bigskip

\begin{proof}
[Proof of Theorem \ref{thm:optimal2}] When $\dfrac{\pi}{2}\leq\theta<\pi$, it
is obvious as $s=0$. Thus assume that $0<\theta<\dfrac{\pi}{2}$.

Let $g$ be such that $g\neq0$ and $\deg(g)=t<s$.  Since $s=\left\lceil
\dfrac{\pi}{\theta}\right\rceil -2$, we have
\[\begin{array}{crcl} & t & < & s \\
\iff & t+2 & \leq & s+1 \\
\iff & (t+2)\theta & \leq & (s+1)\theta \\
\ \, \Longrightarrow & \sin(t+2)\theta & > & 0  \end{array}\]
Let $T_t$ be the $T$-matrix for this multiplier.  Then for all $0 \leq k \leq t$, we have
\[
T_{tk1}=-r^{k+1}\dfrac{\sin(k+2)\theta}{\sin\theta}<0
\]
Then there is no $c \geq 0$ where $cT_t \geq 0$.  Thus, we conclude that$\ \operatorname*{coeffs}\left(
gf\right)  \geq0$ is false.
\end{proof}

\section{A New Poincar\'{e} Multiplier (Theorems \ref{thm:certificate2} and
\ref{thm:coeffsHRM2})}

\label{sec:new}

The linear algebra approach naturally leads to the certificate provided in
Theorem \ref{thm:certificate2}. This certificate for the multiplier is unique
from the one provided originally by Meissner, which will be shown in Section
\ref{sec:compare}.

First, we prove Theorem \ref{thm:certificate2}, the form of a new Poincar\'{e}
multiplier $g$, using linear algebra. We know from Section 3 that we can describe a product $gf$ by giving the matrix product
$c\left[ T_{s} |I\right] $. \bigskip

\begin{proof}
[Proof of Theorem \ref{thm:certificate2}] From Lemma \ref{lem:d2}, we know
$T_{s0},T_{s1}\geq0$ and it suffices to choose $c=\left[
\begin{array}
[c]{cccc}%
0 & \cdots & 0 & 1
\end{array}
\right]  $.  
Now let us find an explicit expression for $g_{R}$.  Let $x_s$ be a column vector of powers of $x$ from $x^0$ to $x^s$.
\begin{align*}
g_{R}  &  =bx_{s}\\
&  =cR_{s}^{-1}x_{s}\ \ \ \ \text{since }c=bR_{s}\\
&  =\left[
\begin{array}
[c]{cccc}%
0 & \cdots & 0 & 1
\end{array}
\right]  R_{s}^{-1}x_{s}\\
&  =\left(  R_{s}^{-1}x_{s}\right)  _{s}\\
&  =\frac{\left\vert
\begin{array}
[c]{ccccc}%
a_{2} &  &  &  & x^{0}\\
a_{1} & \ddots &  &  & \vdots\\
a_{0} &  & \ddots &  & \vdots\\
& \ddots &  & \ddots & \vdots\\
&  & a_{0} & a_{1} & x^{s}%
\end{array}
\right\vert }{\left\vert
\begin{array}
[c]{ccccc}%
a_{2} &  &  &  & \\
a_{1} & \ddots &  &  & \\
a_{0} &  & \ddots &  & \\
& \ddots &  & \ddots & \\
&  & a_{0} & a_{1} & a_{2}%
\end{array}
\right\vert }\mathbb{\ \ }\text{from Cramer's rule}\\
&  =\left\vert
\begin{array}
[c]{ccccc}%
a_{2} &  &  &  & x^{0}\\
a_{1} & \ddots &  &  & \vdots\\
a_{0} &  & \ddots &  & \vdots\\
& \ddots &  & \ddots & \vdots\\
&  & a_{0} & a_{1} & x^{s}%
\end{array}
\right\vert \ \ \text{since }a_{2}=1
\end{align*}

\end{proof}

In order to compare the new multiplier to Meissner's, we'd like to derive a
formula for the coefficients in terms of $r$ and $\theta$. Within the proof of Theorem \ref{thm:coeffsHRM2}, we will provide a motivation for the coefficient formula to be proved. \bigskip

\begin{proof}[Proof of Theorem \ref{thm:coeffsHRM2}]

Recall that
\begin{align*} c & =bR_s \\
\left[\begin{array}{cccc} 0 & \cdots & 0 & 1 \end{array}\right] & =\left[\begin{array}{cccc} b_0 & \cdots & b_{s-1} & b_s \end{array}\right]\left[\begin{array}{ccccc} 1 \\ a_1 & 1 \\ a_0 & a_1 & 1 \\ & \ddots & \ddots & \ddots \\ & & a_0 & a_1 & 1 \end{array}\right] \\
& =\left[\begin{array}{ccccc} b_0+b_1a_1+b_2a_0 & \cdots & b_{s-2}+b_{s-1}a_1+b_sa_0 & b_{s-1}+b_sa_1 & 1 \end{array}\right] \end{align*}
Recall that $b_s=1$.  Then we have 
\begin{align*} b_{s-1} & =-a_1=2r\cos(\theta) \\
& =\dfrac{r\sin(2\theta)}{\sin(\theta)} \\
b_{s-2} & =-b_{s-1}a_1-b_sa_0 \\
& =-\left(\dfrac{r\sin(2\theta)}{\sin(\theta)}\right)\left(-2r\cos(\theta)\right)-r^2\left(\dfrac{\sin(\theta)}{\sin(\theta)}\right) \\
& =\dfrac{2r^2\sin(2\theta)\cos(\theta)-r^2\sin(\theta)}{\sin(\theta)} \\
& = \dfrac{r^2\sin(3\theta)}{\sin(\theta)} \end{align*}
We can see that for $j \geq 2$,
\begin{align*} b_{s-j} & =-b_{s-j+1}a_1-b_{s-j+2}a_0  \end{align*}
Based on $b_{s-1}$ and $b_{s-2}$, we expect that
$$b_{s-j}=r^j \cdot \dfrac{\sin(j+1)\theta)}{\sin(\theta)}.$$
Note we can say $b_{s-j}=\dfrac{T_{j0}}{r^2}$. \bigskip
We will prove the above formula for $b_{s-j}$ holds in general by induction on $j$.
\begin{description}
\item[\textsf{Base Case:}] $j=0$ \\
We have
\begin{align*} 
b_s & =1=r^0 \cdot \dfrac{\sin(\theta)}{\sin(\theta)} \end{align*}
The base case is proved.

\item[\textsf{Inductive Hypothesis:}] Assume for $j \geq 0$ we have $b_{s-j}=r^j \cdot \dfrac{\sin(j+1)\theta}{\sin(\theta)}$.

\item[\textsf{Induction Step:}] Show that 
$$b_{s-(j+1)}=b_{s-j-1}=r^{j+1} \cdot \dfrac{\sin\left(j+2\right)\theta}{\sin(\theta)}.$$
We have
\begin{align*} b_{s-j-1} & =-b_{s-j}a_1-b_{s-j+1}a_0 \\
& =-\left(r^j \cdot \dfrac{\sin(j+1)\theta}{\sin(\theta)}\right)\left(-2r\cos(\theta)\right)-\left(r^{j-1} \cdot \dfrac{\sin(j)\theta}{\sin(\theta)}\right)r^2 \\
& =r^{j+1}\left(\dfrac{2\cos(\theta)\sin(j+1)\theta-\sin(j)\theta}{\sin(\theta)}\right) \\
& =r^{j+1}\left(\dfrac{\sin(j+2)\theta}{\sin(\theta)}\right) \ \ \ \ \text{by the multiple angle formula in Section 6.5.10 of } \cite{Z18} \end{align*}
Hence, the induction step is proved.
\end{description}

\end{proof}

Finally, we will prove that these coefficients are positive.
\begin{lemma}
\label{lem:pos} We have $b_{s-j} \geq0$ for $j=0,\ldots,s$.
\end{lemma}

\begin{proof}
Recall that
\begin{align*}
b_{s-j} & =r^{j} \cdot\dfrac{\sin(j+1)\theta}{\sin(\theta)}
\ \ \ \ \text{from Theorem } \ref{thm:coeffsHRM2}%
\end{align*}
Note
\[%
\begin{array}
[c]{crcl}
& b_{s-j} & \geq & 0\\
\iff & r^{j} \cdot\dfrac{\sin(j+1)\theta}{\sin(\theta)} & \geq & 0\\
\iff & \sin(j+1)\theta) & \geq & 0\ \ \ \text{since } r,\,\sin(\theta)>0\\
\Longleftarrow & 0<(j+1)\theta & \leq & \pi
\end{array}
\]
Since $j \leq s$ and we know
$(s+1)\theta<\pi$ (from the proof of Lemma \ref{lem:d2}), we have $(j+1)\theta\leq(s+1)\theta<\pi$.
\end{proof}

\section{Comparing Meissner's Multiplier with the New Multiplier}

\label{sec:compare}

In this section, we compare the new multiplier $g_R$ with Meissner's in two ways, by
comparing $c$-vectors and by comparing coefficients. Both analyses show that
the new multiplier is an improvement in some way over Meissner's original multiplier.

In order to conduct this comparison fairly, we derive a monic version of
Meissner's Poincar\'{e} multiplier, as the new multiplier is monic.

\begin{lemma}
[Meissner's Monic Multiplier for Degree 2]\label{thm:meissner} Let
$f=x^{2}+a_{1}x+a_{0}\in\mathbb{R}\left[  x\right]  $ be without real roots.
Then a monic witness for $\operatorname*{coeffs}\left(  gf\right)  \geq0$ is
given by
\[
g_{M}^{*}=\sum_{i=0}^{s} \dfrac{r^{s-i}\, \sin(i+1)\theta}{\sin(s+1)\theta
}\ x^{i}=r^{s}\sum_{i=0}^{s} \dfrac{\sin(i+1)\theta}{\sin(s+1)\theta}\ \left(
\dfrac{x}{r}\right)  ^{i}
\]
where $s=\left\lceil \dfrac{\pi}{\theta} \right\rceil -2$ and $re^{\pm
i\theta}$ are the non-real roots of $f$.
\end{lemma}

\begin{proof}
Recall Meissner's multiplier from Formula \ref{eq:Meissner} \cite{M11}:%

\[
g_{M}=\sum_{i=0}^{s} \dfrac{r^{2}\sin(i+1)\theta}{\sin(\theta)}\left(
\dfrac{x}{r}\right)  ^{i}%
\]
where
\begin{align*}
s  &  =\left\lceil \dfrac{\pi}{\theta} \right\rceil -2.
\end{align*}
This multiplier can be made monic by dividing all terms by the leading coefficient.%

\begin{align*}
g_{M}^{*}  &  =\dfrac{1}{\dfrac{r^{2}\sin(s+1)\theta}{r^{s}\,\sin(\theta)}%
}\ \ \sum_{i=0}^{s} \dfrac{r^{2}\sin(i+1)\theta}{\sin(\theta)}\left(
\dfrac{x}{r}\right)  ^{i}\\
&  =\dfrac{r^{s}\,\sin(\theta)}{r^{2}\sin(s+1)\theta}\ \ \sum_{i=0}^{s}
\dfrac{r^{2}\sin(i+1)\theta}{\sin(\theta)}\left(  \dfrac{x}{r}\right)  ^{i}\\
&  =\sum_{i=0}^{s} \dfrac{r^{s}\, \sin(i+1)\theta}{\sin(s+1)\theta}\left(
\dfrac{x}{r}\right)  ^{i}\\
&  =\sum_{i=0}^{s} \dfrac{r^{s-i}\, \sin(i+1)\theta}{\sin(s+1)\theta}\ x^{i}\\
&  =r^{s}\sum_{i=0}^{s} \dfrac{\sin(i+1)\theta}{\sin(s+1)\theta}\ \left(
\dfrac{x}{r}\right)  ^{i}%
\end{align*}

\end{proof}

\begin{example}
Recall the polynomial from Example \ref{ex:notequal}, $f=x^{2}-2\cos\left(
\dfrac{2\pi}{7}\right) x+1$, which has non-real roots $1e^{\pm\frac{2\pi}{7}%
i}$ and $s=2$. Recall that the new multiplier, $g_{R}$, is given by
\begin{align*}
g_{R}  &  =x^{2}+2\cos\left( \dfrac{2\pi}{7}\right) x+\left( 4\cos^{2}\left(
\dfrac{2\pi}{7}\right) -1\right) \\
&  \approx x^{2}+1.247x+0.555
\end{align*}

The monic Meissner multiplier, $g_{M}^{*}$, is given by
\begin{align*}
g_{M}^{*}  &  =\sum_{i=0}^{2} \dfrac{1^{2-i}\, \sin(i+1)\theta}{\sin
(2+1)\dfrac{2\pi}{7}}\ x^{i}\\
&  = \dfrac{\sin\left( \dfrac{6\pi}{7}\right) }{\sin\left( \dfrac{6\pi}%
{7}\right) }x^{2}+\dfrac{\sin\left( \dfrac{4\pi}{7}\right) }{\sin\left(
\dfrac{6\pi}{7}\right) }x+\dfrac{\sin\left( \dfrac{2\pi}{7}\right) }%
{\sin\left( \dfrac{6\pi}{7}\right) }\\
&  \approx x^{2}+2.467x+1.802
\end{align*}

Although both multipliers are monic, it is still true that $g_{M}^{*} \neq
g_{R}$.
\end{example}

Naturally, we wonder when these two monic multipliers are the same and when
they are different. As shown below, when they are different, the multiplier
from Theorem \ref{thm:certificate2} has advantages over Meissner's.

\subsection{Comparison via $c$-Vector}

We will recover the $c$-vector such that $\operatorname*{coeffs}\left(  gf\right)=c[T_s|I]$ for
Meissner's monic multiplier in order to compare the two. Recall that the
$c$-vector for the multiplier $g_R$ from Theorem \ref{thm:certificate2} is
$c_{R}=\left[
\begin{array}
[c]{cccc}%
0 & \cdots & 0 & 1
\end{array}
\right]  $.

\begin{lemma}
The $c$-vector for Meissner's monic Poincar\'{e} multiplier is
\[
c_{M}=\left[
\begin{array}
[c]{ccccc}%
0 & \cdots & 0 & r^{2}\, \dfrac{T_{s1}}{T_{s0}} & 1
\end{array}
\right] .
\]

\end{lemma}

\begin{proof}
From Lemma \ref{thm:meissner}, for $f=x^{2}-2r\cos(\theta)+r^{2}$ we have
\[
g_{M}^{*}=\sum_{i=0}^{s}\dfrac{r^{s-i}\,\sin(i+1)\theta}{\sin(s+1)\theta
}\ x^{i}.
\]
Then
\begin{align*}
c  &  =b\,R_{s} \\
&  =\left[
\begin{array}
[c]{cccc}%
\dfrac{r^{s}\,\sin(\theta)}{\sin(s+1)\theta} & \dfrac{r^{s-1}\,\sin(2\theta
)}{\sin(s+1)\theta} & \cdots & \dfrac{r^{0}\,\sin(s+1)\theta}{\sin(s+1)\theta}%
\end{array}
\right]  \left[
\begin{array}
[c]{ccccc}%
1 &  &  &  & \\
-2r\cos(\theta) & \ddots &  &  & \\
r^{2} & \ddots & \ddots &  & \\
& \ddots & \ddots & \ddots & \\
&  & r^{2} & -2r\cos(\theta) & 1
\end{array}
\right]  .
\end{align*}
We examine the entries of Meissner's $c$-vector.

\begin{description}
\item[\textsf{Case $i=s:$}]
\[
c_{s}=1\cdot\dfrac{r^{0}\,\sin(s+1)\theta}{\sin(s+1)\theta}=1
\]

\item[\textsf{Case $i=s-1:$}]
\begin{align*}
c_{s-1}  &  =1\cdot\dfrac{r^{1}\,\sin((s-1)+1)\theta}{\sin(s+1)\theta}%
-2r\cos(\theta)\cdot\dfrac{r^{0}\,\sin(s+1)\theta}{\sin(s+1)\theta}\\
&  =\dfrac{r\,\sin(s\,\theta)}{\sin(s+1)\theta}-2r\cos(\theta)\\
&  =r\left(  \dfrac{\sin(s\,\theta)}{\sin(s+1)\theta}-2\cos(\theta)\right) \\
&  =r\dfrac{\sin(s\,\theta)-2\cos(\theta)\sin(s+1)\theta}{\sin(s+1)\theta}\\
&  =r\dfrac{\sin(s\,\theta)-2\cos(\theta)\left[  \sin(s\,\theta)\,\cos
(\theta)+\cos(s\,\theta)\,\sin(\theta)\right]  }{\sin(s+1)\theta}\\
&  =r\dfrac{\sin(s\,\theta)-\left(  2\sin(s\,\theta)\,\cos^{2}(\theta
)+2\sin(\theta)\,\cos(\theta)\,\cos(s\,\theta)\right)  }{\sin(s+1)\theta}\\
&  =r\dfrac{\sin(s\,\theta)-\left(  \sin(s\,\theta)(1+\cos(2\theta
))+\sin(2\theta)\,\cos(s\,\theta)\right)  }{\sin(s+1)\theta}\\
&  =r\dfrac{-\left(  \sin(s\,\theta)\,\cos(2\theta)+\sin(2\theta
)\,\cos(s\,\theta)\right)  }{\sin(s+1)\theta}\\
&  =-r\frac{\sin\left(  s+2\right)  \theta}{\sin(s+1)\theta}\\
&  =r^{2}\, \frac{T_{s1}}{T_{s0}}\ \ \ \ \text{from Lemma } \ref{lem:Ast2}%
\end{align*}

\item[\textsf{Case $i\leq s-2:$}]
\begin{align*}
c_{i}  &  =1\cdot\dfrac{r^{s-i}\,\sin(i+1)\theta}{\sin(s+1)\theta}%
-2r\cos(\theta)\cdot\dfrac{r^{s-(i+1)}\,\sin((i+1)+1)\theta}{\sin(s+1)\theta
}+r^{2}\cdot\dfrac{r^{s-(i+2)}\,\sin((i+2)+1)\theta}{\sin(s+1)\theta}\\
&  =\dfrac{r^{s-i}\,\sin(i+1)\theta}{\sin(s+1)\theta}-\dfrac{2r^{s-i}%
\,\cos(\theta)\,\sin(i+2)\theta}{\sin(s+1)\theta}+\dfrac{r^{s-i}%
\,\sin(i+3)\theta}{\sin(s+1)\theta}\\
&  =\dfrac{r^{s-i}}{\sin(s+1)\theta}\left(  \sin(i+1)\theta-2\cos
(\theta)\,\sin(i+2)\theta+\sin(i+3)\theta\right) \\
&  =\dfrac{r^{s-i}}{\sin(s+1)\theta}\left(  \sin(i+2)\theta\,\cos(\theta
)-\cos(i+2)\theta\,\sin(\theta)-2\cos(\theta)\,\sin(i+2)\theta+\sin
(i+2)\theta\,\cos(\theta)+\cos(i+2)\theta\,\sin(\theta)\right) \\
&  =0
\end{align*}

\end{description}

Hence, we have $c_{M}=\left[
\begin{array}
[c]{ccccc}%
0 & \cdots & 0 & r^{2}\, \dfrac{T_{s1}}{T_{s0}} & 1
\end{array}
\right] .$
\end{proof}

In comparison, the new multiplier $g_R$  only
contains one non-zero entry rather than two. \bigskip

Now that we have both $c$-vectors, we can determine when the two multipliers are equal.

\begin{corollary}
\label{cor:equal} We have $c_{M}=c_{R}$ when $\theta=\dfrac{\pi}{s+2}$.
\end{corollary}

\begin{proof}
Comparing $c$-vectors, the new Poincar\'{e} multiplier is equal to Meissner's
when $r^{2}\,\dfrac{T_{s1}}{T_{s0}}=0$. Note that
\[%
\begin{array}
[c]{crcl}
& r^{2}\,\dfrac{T_{s1}}{T_{s0}} & = & 0 \vspace{0.1cm}\\
\iff & \dfrac{\sin(s+2)\theta}{\sin(s+1)\theta} & = & 0 \ \ \ \ \text{from
Lemma } \ref{lem:Ast2} \text{ and since } r>0\\
\iff & \sin(s+2)\theta & = & 0 \ \ \ \ \text{since } \sin(s+1)\theta \neq 0 \text{
from the proof of Lemma } \ref{lem:d2}\\
\iff & (s+2)\theta & = & \pi\\
\iff & \theta & = & \dfrac{\pi}{s+2}%
\end{array}
\]

\end{proof}
\begin{remark}[Comparison via Partial Ordering on Vectors]
In the vector space $\mathbb{R}^{s+1}$, we can say $c_R \leq c_M$ since $c_{R,s-1} \leq c_{M,s-1}$ and all other entries are equal.
\end{remark}

\subsection{Comparison via the Coefficients in terms of $r$ and $\theta$}

Now that we have compared Meissner's multiplier $g_M^*$ to the new $g_R$ via our linear
algebra framework, we will compare the two via the coefficients in terms
of $r$ and $\theta$. Recall from Remark \ref{rem:coeff} that the coefficients
of the new multiplier are given by
\begin{align*}
\operatorname*{coeff}(g_{R},i)  &  =\dfrac{r^{s-i}\, \sin(s-i+1)\theta}%
{\sin(\theta)}.
\end{align*}
where $s=\left\lceil \dfrac{\pi}{\theta} \right\rceil -2$ and $re^{\pm
i\theta}$ are the non-real roots of $f$. From Lemma \ref{thm:meissner}, the
coefficients of Meissner's monic multiplier are given by
\[
\operatorname*{coeff}(g_{M}^{*},i)=\dfrac{r^{s-i}\, \sin(i+1)\theta}%
{\sin(s+1)\theta}.
\]

\begin{lemma} We have for $i=0,\hdots,s$
$$\operatorname*{coeff}(g_{R},i) \leq \operatorname*{coeff}(g_{M}^*,i).$$
\end{lemma}

\begin{proof}
Note that
\[\begin{array}{crcl} & \operatorname*{coeff}(g_{R},i) & \leq & \operatorname*{coeff}(g_{M}^*,i) \\
\iff & \dfrac{r^{s-i}\, \sin(s-i+1)\theta}%
{\sin(\theta)} & \leq & \dfrac{r^{s-i}\, \sin(i+1)\theta}%
{\sin(s+1)\theta} \\
\iff & \sin(s+1)\theta\sin(s-i+1)\theta-\sin(i+1)\theta\sin(\theta) & \leq & 0  \end{array}\]
since both $\sin(\theta)$ and $\sin(s+1)\theta)$ are positive.  We have the following identities:
\begin{align}
\sin A\sin B  & =\frac{1}{2}\left(  \cos\left( A-B
\right)  -\cos\left( A+B\right)  \right) \label{sc}   \\
\frac{1}{2}\left(  \cos P-\cos Q\right) & =\sin\left(  \frac{Q+P}{2}\right)  \sin\left(  \frac{Q-P}{2}\right) \label{cs}  
\end{align}
If we label
\begin{align*} 
A_1 & =(s+1)\theta \\
B_1 & =(s-i+1)\theta \\
A_2 & =(i+1)\theta \\
B_2 & =\theta \end{align*}
we can easily verify that%
\[
A_{1}-B_{1}=A_{2}-B_{2}%
\]
Then
\begin{align*} & \sin(A_1)\sin(B_1)-\sin(A_2)\sin(B_2)   \\
=\ & \frac{1}{2}\left(  \cos\left(  A_{1}-B_{1}\right)
-\cos\left(  A_{1}+B_{1}\right)  \right) -\frac{1}{2}\left(  \cos\left(  A_{2}-B_{2}\right)
-\cos\left(  A_{2}+B_{2}\right)  \right) \ \ \text{from } (\ref{sc}) \\
=\ & \frac{1}{2}\left( \cos\left(  A_{2}+B_{2}\right) - \cos\left(  A_{1}+B_{1}\right) \right)  \\
=\ & \sin\left(  \frac{\left(  A_{1}+B_{1}\right)+\left(  A_{2}+B_{2}\right)
}{2}\right)  \sin\left(  \frac{\left(  A_{1}+B_{1}\right)  -\left(
A_{2}+B_{2}\right)  }{2}\right) \ \ \text{from } (\ref{cs}) \\
=\ & \sin\left(  \frac{\left(  A_{1}+B_{1}\right)+\left(  A_1-B_1+B_2+B_{2}\right)
}{2}\right)  \sin\left(  \frac{\left(  A_{1}+B_{1}\right)  -\left(
A_{2}+B_{2}\right)  }{2}\right) \ \ \text{since } A_2=A_1-B_1+B_2 \\
=\ & \sin\left( A_1+B_2\right)  \sin\left(  \frac{\left(  A_{1}+A_1-A_2+B_2\right)  -\left(
A_{2}+B_2\right)  }{2}\right) \ \ \text{since } B_1=A_1-A_2+B_2 \\
=\ & \sin\left(  A_{1}+B_{2}\right)  \sin\left(  A_{1}-A_{2}\right)  \\
=\ & \sin\left(  \left(  s+1\right)  \theta+\theta\right)  \sin\left(  \left(
s+1\right)  \theta-\left(  i+1\right)  \theta\right)  \\
=\ & \sin\left(    s+2\right)  \theta  \sin\left(
s-i\right)  \theta
\end{align*}
Since $\sin(s+2)\theta \leq 0$ and $\sin(s-i)\theta>0$, we have 
$$\sin(s+1)\theta\sin(s-i+1)\theta-\sin(i+1)\theta\sin(\theta)= \sin\left(    s+2\right)  \theta  \sin\left(
s-i\right)  \theta \leq 0.$$
\end{proof}

We investigate this relationship between the coefficient expressions further. We define the ratio between the coefficients to be
\[
\operatorname*{ratio}\left( r,\theta,i\right) =\dfrac{\operatorname*{coeff}%
(g_{R},i)}{\operatorname*{coeff}(g_{M}^{*},i)}=\dfrac{\sin(s-i+1)\theta\, \sin(s+1)\theta}{\sin(\theta)\, \sin
(i+1)\theta}.
\]

Note that $\operatorname*{ratio}\left( r,\theta,i\right) $ is independent
of $r$, so we may simplify to $\operatorname*{ratio}\left( \theta,i\right) $.
We expect from Corollary \ref{cor:equal} that $\operatorname*{ratio}\left(
\theta,i\right) =1$ when $\theta=\dfrac{\pi}{s+2}$, or when $\theta$ can be
written as $\pi$ divided by an integer. Below, we include several graphs
showing $\operatorname*{ratio}\left( \theta,i\right) $ for each $i$ from 0 to
$s-1$ and for ten angles equally spaced between $\dfrac{\pi}{s+2} \leq\theta<
\dfrac{\pi}{s+1}$. In both graphs, the radius is assumed to be 1, as any choice
of radius will do. Finally, note that we did not graph $\operatorname*{ratio}%
\left( \theta,s\right) $ as this is always equal to 1 since both multipliers
are monic.

In each graph, the ten angles are given by
\[
\theta_{i}=\left( 1-\dfrac{i}{10}\right)  \cdot\dfrac{\pi}{s+2}+\dfrac{i}{10}
\cdot\dfrac{\pi}{s+1},\ i=0,\ldots,9.
\]
\begin{figure}[h]
\centering
\begin{tikzpicture}
\node (myfirstpic) at (-4,0) {\includegraphics[width=2.5in]{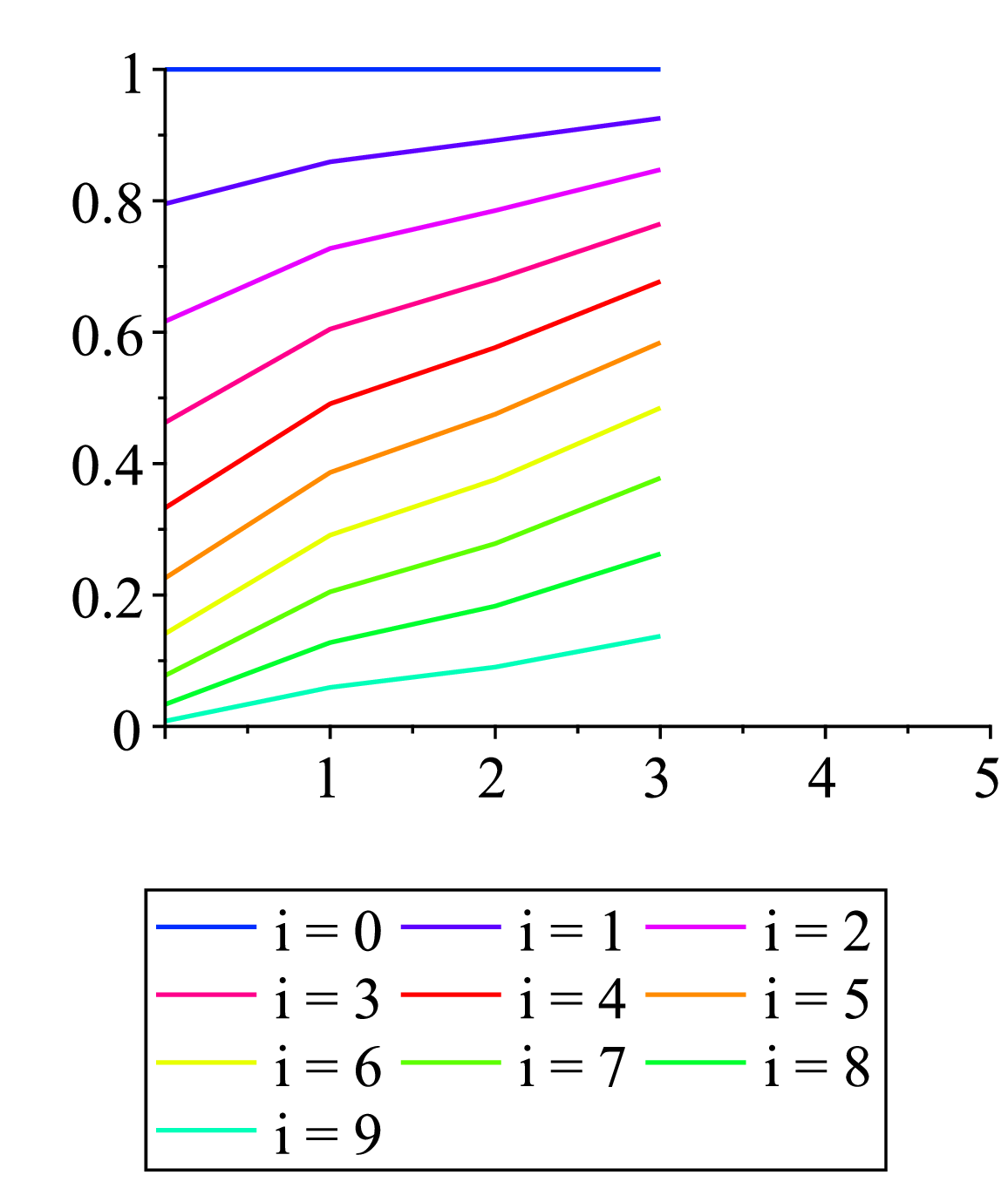}};
\node[right] at (0,3) {In this example, $\dfrac{\pi}{6} \leq\theta<\dfrac{\pi}{5}$ or $s=4$.};
\node[right] at (0,2) {The graph shows $\operatorname*{ratio}\left(\theta,0\right),\ldots,\operatorname*{ratio}\left(\theta,3\right)$.};
\node[right] at (0,1) {We used ten angles between $\dfrac{\pi}{6}$ and $\dfrac{\pi}{5}$ defined by};
\node[right] at (0,0.2) {$\theta_i=\left(1-\dfrac{i}{10}\right) \cdot \dfrac{\pi}{6}+\dfrac{i}{10} \cdot \dfrac{\pi}{5},\ i=0,\ldots,9.$};
\end{tikzpicture}
\caption{Comparison of the Ratio of Coefficients for $s=4$}%
\end{figure}\bigskip

\begin{figure}[h]
\centering
\begin{tikzpicture}
\node (myfirstpic) at (-4,0) {\includegraphics[width=2.5in]{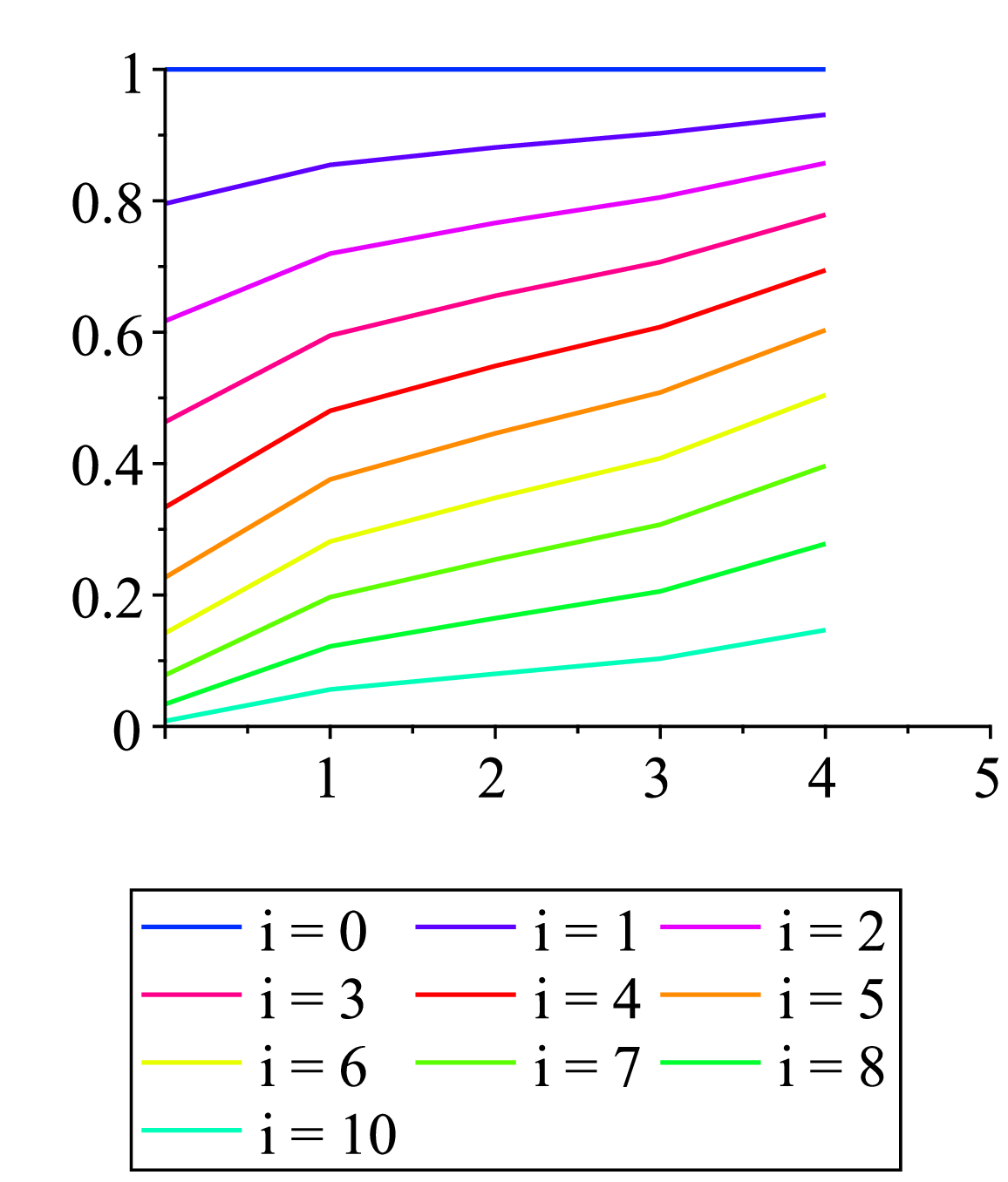}};
\node[right] at (0,3) {In this example, $\dfrac{\pi}{7} \leq\theta<\dfrac{\pi}{6}$ or $s=5$.};
\node[right] at (0,2) {The graph shows $\operatorname*{ratio}\left(\theta,0\right),\ldots,\operatorname*{ratio}\left(\theta,4\right)$.};
\node[right] at (0,1) {We used ten angles between $\dfrac{\pi}{7}$ and $\dfrac{\pi}{6}$ defined by};
\node[right] at (0,0.2) {$\theta_i=\left(1-\dfrac{i}{10}\right) \cdot \dfrac{\pi}{7}+\dfrac{i}{10} \cdot \dfrac{\pi}{6},\ i=0,\ldots,9.$};
\end{tikzpicture}
\caption{Comparison of the Ratio of Coefficients for $s=5$}%
\end{figure}\bigskip

These graphs confirm Corollary \ref{cor:equal}, showing that the coefficients
are equal when $\theta=\dfrac{\pi}{s+2}$. The new multiplier has smaller
coefficients for the angles $\dfrac{\pi}{s+2}<\theta<\dfrac{\pi}{s+1}$. The
coefficients of the new multiplier are significantly smaller than Meissner's
coefficients for those angles close to but not equal to $\dfrac{\pi}{s+1}$.

\begin{lemma}
We have $\displaystyle\lim_{\theta\to\frac{\pi}{s+1}}\ \operatorname{ratio}%
\left( \theta,i\right) =0$ for all $i$.
\end{lemma}

\begin{proof}
Note that
\begin{align*}
\operatorname{ratio}\left( \theta,i\right)   &  =\dfrac{\sin(s-i+1)\theta\,
\sin(s+1)\theta}{\sin(\theta)\, \sin(i+1)\theta}%
\end{align*}
Then
\begin{align*}
\lim_{\theta\to\frac{\pi}{s+1}}\ \operatorname{ratio}\left( \theta,i\right)
&  =\lim_{\theta\to\frac{\pi}{s+1}}\ \dfrac{\sin(s-i+1)\theta\, \sin
(s+1)\theta}{\sin(\theta)\, \sin(i+1)\theta}\\
&  =\left( \lim_{\theta\to\frac{\pi}{s+1}}\ \dfrac{\sin(s-i+1)\theta}%
{\sin(\theta)\, \sin(i+1)\theta}\right)  \cdot\left( \lim_{\theta\to\frac{\pi
}{s+1}}\ \sin(s+1)\theta\right) 
\end{align*}
Note that
\[
\lim_{\theta\to\frac{\pi}{s+1}}\ \sin(s+1)\theta=\sin(\pi)=0.
\]
Recall that $i \leq s$, so $(i+1)\theta\leq(s+1)\theta<\pi$. Hence, we know
\[
\lim_{\theta\to\frac{\pi}{s+1}}\ \dfrac{\sin(s-i+1)\theta}{\sin(\theta)\,
\sin(i+1)\theta}
\]
is defined and finite. Thus
\[
\lim_{\theta\to\frac{\pi}{s+1}}\ \operatorname{ratio}\left( \theta,i\right)
=0.
\]

\end{proof}

\bibliographystyle{abbrv}
\bibliography{../../references/all}

\end{document}